\documentclass[12pt,reqno]{amsart}
\usepackage{amscd,amssymb,amsmath,multicol,float,scalefnt,cancel}
\restylefloat{table}
\newtheorem{thm}[equation]{Theorem}
\numberwithin{equation}{section}
\newtheorem{cor}[equation]{Corollary}

\newtheorem{lem}[equation]{Lemma}

\newtheorem{prop}[equation]{Proposition}

\newtheorem{tab}[equation]{Table}

\begin{document}
\raggedbottom \voffset=-.7truein \hoffset=0truein \vsize=8truein
\hsize=6truein \textheight=8truein \textwidth=6truein
\baselineskip=18truept

\def\mapright#1{\ \smash{\mathop{\longrightarrow}\limits^{#1}}\ }
\def\mapleft#1{\smash{\mathop{\longleftarrow}\limits^{#1}}}
\def\mapup#1{\Big\uparrow\rlap{$\vcenter {\hbox {$#1$}}$}}
\def\mapdown#1{\Big\downarrow\rlap{$\vcenter {\hbox {$\ssize{#1}$}}$}}
\def\mapne#1{\nearrow\rlap{$\vcenter {\hbox {$#1$}}$}}
\def\mapse#1{\searrow\rlap{$\vcenter {\hbox {$\ssize{#1}$}}$}}
\def\mapr#1{\smash{\mathop{\rightarrow}\limits^{#1}}}
\def\ss{\smallskip}
\def\vp{v_1^{-1}\pi}
\def\at{{\widetilde\alpha}}
\def\sm{\wedge}
\def\la{\langle}
\def\ra{\rangle}
\def\on{\operatorname}
\def\ol#1{\overline{#1}{}}
\def\spin{\on{Spin}}
\def\lbar{ \ell}
\def\qed{\quad\rule{8pt}{8pt}\bigskip}
\def\ssize{\scriptstyle}
\def\a{\alpha}
\def\bz{{\Bbb Z}}
\def\im{\on{im}}
\def\ct{\widetilde{C}}
\def\ext{\on{Ext}}
\def\sq{\on{Sq}}
\def\eps{\epsilon}
\def\ar#1{\stackrel {#1}{\rightarrow}}
\def\br{{\bold R}}
\def\bC{{\bold C}}
\def\bA{{\bold A}}
\def\bB{{\bold B}}
\def\bD{{\bold D}}
\def\bh{{\bold H}}
\def\bQ{{\bold Q}}
\def\bP{{\bold P}}
\def\bx{{\bold x}}
\def\bo{{\bold{bo}}}
\def\si{\sigma}
\def\Ebar{{\overline E}}
\def\dbar{{\overline d}}
\def\Sum{\sum}
\def\tfrac{\textstyle\frac}
\def\tb{\textstyle\binom}
\def\Si{\Sigma}
\def\w{\wedge}
\def\equ{\begin{equation}}
\def\b{\beta}
\def\G{\Gamma}
\def\g{\gamma}
\def\k{\kappa}
\def\psit{\widetilde{\Psi}}
\def\tht{\widetilde{\Theta}}
\def\psiu{{\underline{\Psi}}}
\def\thu{{\underline{\Theta}}}
\def\aee{A_{\text{ee}}}
\def\aeo{A_{\text{eo}}}
\def\aoo{A_{\text{oo}}}
\def\aoe{A_{\text{oe}}}
\def\vbar{{\overline v}}
\def\endeq{\end{equation}}
\def\sn{S^{2n+1}}
\def\zp{\bold Z_p}
\def\A{{\cal A}}
\def\P{{\mathcal P}}
\def\cj{{\cal J}}
\def\zt{{\bold Z}_2}
\def\bs{{\bold s}}
\def\bof{{\bold f}}
\def\bq{{\bold Q}}
\def\be{{\bold e}}
\def\Hom{\on{Hom}}
\def\ker{\on{ker}}
\def\coker{\on{coker}}
\def\da{\downarrow}
\def\colim{\operatornamewithlimits{colim}}
\def\zphat{\bz_2^\wedge}
\def\io{\iota}
\def\Om{\Omega}
\def\Prod{\prod}
\def\e{{\cal E}}
\def\exp{\on{exp}}
\def\abar{{\overline a}}
\def\xbar{{\overline x}}
\def\ybar{{\overline y}}
\def\zbar{{\overline z}}
\def\Rbar{{\overline R}{}}
\def\nbar{{\overline n}}
\def\cbar{{\overline c}}
\def\qbar{{\overline q}}
\def\bbar{{\overline b}}
\def\et{{\widetilde E}}
\def\ni{\noindent}
\def\coef{\on{coef}}
\def\den{\on{den}}
\def\lcm{\on{l.c.m.}}
\def\vi{v_1^{-1}}
\def\ot{\otimes}
\def\psibar{{\overline\psi}}
\def\mhat{{\hat m}}
\def\exc{\on{exc}}
\def\ms{\medskip}
\def\ehat{{\hat e}}
\def\etao{{\eta_{\text{od}}}}
\def\etae{{\eta_{\text{ev}}}}
\def\dirlim{\operatornamewithlimits{dirlim}}
\def\gt{\widetilde{L}}
\def\lt{\widetilde{\lambda}}
\def\st{\widetilde{s}}
\def\ft{\widetilde{f}}
\def\sgd{\on{sgd}}
\def\lfl{\lfloor}
\def\rfl{\rfloor}
\def\ord{\on{ord}}
\def\gd{{\on{gd}}}
\def\rk{{{\on{rk}}_2}}
\def\nbar{{\overline{n}}}
\def\lg{{\on{lg}}}
\def\cR{\mathcal{R}}
\def\cS{\mathcal{S}}
\def\cT{\mathcal{T}}
\def\N{{\Bbb N}}
\def\Z{{\Bbb Z}}
\def\Q{{\Bbb Q}}
\def\R{{\Bbb R}}
\def\C{{\Bbb C}}
\def\l{\left}
\def\r{\right}
\def\mo{\on{mod}}
\def\xt{\times}
\def\notimm{\not\subseteq}
\def\Remark{\noindent{\it  Remark}}

\def\*#1{\mathbf{#1}}
\def\0{$\*0$}
\def\1{$\*1$}
\def\22{$(\*2,\*2)$}
\def\33{$(\*3,\*3)$}
\def\ss{\smallskip}
\def\ssum{\sum\limits}
\def\dsum{\displaystyle\sum}
\def\la{\langle}
\def\ra{\rangle}
\def\on{\operatorname}
\def\o{\on{o}}
\def\U{\on{U}}
\def\lg{\on{lg}}
\def\a{\alpha}
\def\bz{{\Bbb Z}}
\def\eps{\varepsilon}
\def\bc{{\bold C}}
\def\bN{{\bold N}}
\def\nut{\widetilde{\nu}}
\def\tfrac{\textstyle\frac}
\def\b{\beta}
\def\G{\Gamma}
\def\g{\gamma}
\def\zt{{\Bbb Z}_2}
\def\zth{{\bold Z}_2^\wedge}
\def\bs{{\bold s}}
\def\bx{{\bold x}}
\def\bof{{\bold f}}
\def\bq{{\bold Q}}
\def\be{{\bold e}}
\def\lline{\rule{.6in}{.6pt}}
\def\xb{{\overline x}}
\def\xbar{{\overline x}}
\def\ybar{{\overline y}}
\def\zbar{{\overline z}}
\def\ebar{{\overline \be}}
\def\nbar{{\overline n}}
\def\rbar{{\overline r}}
\def\Mbar{{\overline M}}
\def\et{{\widetilde e}}
\def\ni{\noindent}
\def\ms{\medskip}
\def\ehat{{\hat e}}
\def\xhat{{\widehat x}}
\def\nbar{{\overline{n}}}
\def\minp{\min\nolimits'}
\def\N{{\Bbb N}}
\def\Z{{\Bbb Z}}
\def\Q{{\Bbb Q}}
\def\R{{\Bbb R}}
\def\C{{\Bbb C}}
\def\notint{\cancel\cap}
\def\el{\ell}
\def\TC{\on{TC}}
\def\dstyle{\displaystyle}
\def\ds{\dstyle}
\def\Remark{\noindent{\it  Remark}}
\title[Topological complexity of polygon spaces with small code]
{Topological complexity of planar polygon spaces with small genetic code}
\author{Donald M. Davis}
\address{Department of Mathematics, Lehigh University\\Bethlehem, PA 18015, USA}
\email{dmd1@lehigh.edu}
\date{January 19, 2016}


\maketitle
\begin{abstract} We determine lower bounds for the topological complexity of many planar polygon spaces mod isometry. With very few exceptions, the upper and lower bounds given by dimension and cohomology considerations differ by 1. This is true for 130 of the 134 generic 7-gon spaces. Our results apply to spaces of $n$-gons for all $n$, but primarily for
 those whose genetic codes, in the sense of Hausmann and Rodriguez, are moderately small.
 \end{abstract}

\section{Statement of results}\label{intro}

The topological complexity, $\TC(X)$, of a topological space $X$ is, roughly, the number of rules required to specify how to move between any two points of $X$. A ``rule'' must be such that the choice of path varies continuously with the choice of endpoints. (See \cite[\S4]{F}.) We study $\TC(X)$ where $X=\Mbar(\lbar)$ is the space of equivalence classes of oriented $n$-gons in  the plane with consecutive sides of length $\ell_1,\ldots,\ell_n$, identified under translation, rotation, and reflection. (See, e.g., \cite[\S6]{HK}.) Here $\lbar=(\ell_1,\ldots,\ell_n)$ is an $n$-tuple of positive real numbers. Thus
 $$\Mbar(\lbar)=\{(z_1,\ldots,z_n)\in (S^1)^n:\ell_1z_1+\cdots+\ell_nz_n=0\}/O(2).$$
We can think of the sides of the polygon as linked arms of a robot, and then $\TC(X)$ is the number of rules required to program the robot to move from any configuration to any other.

Since permuting the $\ell_i$'s does not affect the space up to homeomorphism, we may assume $\ell_1\le\cdots\le\ell_n$. We assume that  $\lbar$ is {\it generic}, which means that there is no subset $M\subset[n]=\{1,\ldots,n\}$ with $\ds\sum_{i\in M}\ell_i=\tfrac12\dsum_{i=1}^n\ell_i$. When $\lbar$ is generic, $\Mbar(\lbar)$ is an
$(n-3)$-manifold (\cite[p.314]{HK}), and hence, by \cite[Cor 4.15]{F}, satisfies
\begin{equation}\label{ineq1}\TC(\Mbar(\lbar))\le 2n-5.\end{equation}
When $\ell$ is generic, $\Mbar(\ell)$ contains no straight-line polygons. We also assume that $\ell_n<\ell_1+\cdots+\ell_{n-1}$, so that $\Mbar(\ell)$ is nonempty.

It is well-understood that the homeomorphism type of $\Mbar(\lbar)$ is determined by which subsets $S$ of $[n]$ are {\it short}, which means that
$\dsum_{i\in S}\ell_i<\tfrac12\dsum_{i=1}^n\ell_i$. For generic $\lbar$, a subset of $[n]$ which is not short is called {\it long}. Define a partial order on the power set of $[n]$  by $S\le T$ if  $S=\{s_1,\ldots,s_\ell\}$ and $T\supset\{t_1,\ldots,t_\ell\}$ with $s_i\le t_i$ for all $i$.
As introduced in \cite{H}, the {\it genetic code} of $\lbar$  is the  set of maximal elements (called {\it genes}) in the set of short subsets of $[n]$ which contain $n$.
The homeomorphism type of $\Mbar(\lbar)$ is determined by the genetic code of $\lbar$. A list of all genetic codes for $n\le 9$ appears in \cite{web}. For $n=6$, 7, and 8, there are 20, 134, and 2469 genetic codes, respectively.

Our main theorem is a lower bound for $\TC(\Mbar(\ell))$ just 1 less than the upper bound in (\ref{ineq1}) for almost all $\ell$'s whose genetic code consists of a fairly small number of fairly small sets.
\begin{thm}\label{mainthm} Let $\ell=(\ell_1,\ldots,\ell_n)$ with $n\ge5$. If the sizes of the genes of $\ell$ are any of those listed in Table \ref{gT}, and the genetic code of $\ell$ is not $\la521\ra$, $\la6321\ra$, $\la 7321\ra$, or $\la7521\ra$, then
$$\TC(\Mbar(\ell))\ge 2n-6.$$
\end{thm}

Here and throughout, we write genes consisting of 1-digit numbers, or $n$ followed by 1-digit numbers, by concatenating those digits. We now explain the ``T'' (Type) notation that appears in Table \ref{gT}.
We say that a genetic code $\la S_1,S_2\ra$  has Type 1 ($T_1$) if $1\in S_1\cap S_2$.
To describe genetic codes of Type 2, it is useful to introduce the notion of {\it gees}, which are genes without the $n$. This useful terminology will pervade the paper. For example, since $\la n31\ra$ is a genetic code for $n\ge 6$, the associated gee is $31$. Since all genes for $\ell=(\ell_1,\ldots,\ell_n)$ include the element $n$, it is worthwhile for many reasons to omit writing the $n$. In Table \ref{gT}, the associated gees all have size 1 less than the listed gene sizes. For genetic codes of the indicated sizes, we say they have Type 2 ($T_2$) if they are not of Type 1 and  their gees are those of a genetic code with $n=7$.
We considered these so that we could complete the case $n=7$. For example, $\la 7521,762\ra$ is a genetic code, so $\la n521,n62\ra$ appears in our table as a Type-2 code of sizes 4 and 3. On the other hand $\la 7521,763\ra$ is not a genetic code (because $\{6,4,3\}$ would be both long and short) while $\la 8521,863\ra$ is a genetic code, and so the analysis of codes $\la n521,n63\ra$ for $n\ge8$ does not appear in our table. Certainly it could be handled by our methods, but we had to stop somewhere.

\begin{tab}\label{gT}

\begin{center}
{\scalefont{1}{
$$\renewcommand\arraystretch{1.0}\begin{array}{c|cccc}
&\multicolumn{4}{c}{\text{Number of occurrences}}\\
\text{Gene sizes}&n=5&n=6&n=7&n=8\\
\hline
2&4&5&6&7\\
3&&5&15&21\\
4&&&4&21\\
3,3&&&15&35\\
4,3\ T_1&&&8&20\\
4,3\ T_2&&&10&10\\
3,3,3\ T_2&&&1&1\\
4,3,3\ T_2&&&14&14\\
4,3,3,3\ T_2&&&2&2\\
\text{anything},2&&8&55&559
\end{array}$$}}
\end{center}
\end{tab}

The cases 521 and 6321 excluded in Theorem \ref{mainthm} are homeomorphic to tori $(S^1)^2$ and $(S^1)^3$. It is known (\cite[(4.12)]{F}) that $\TC((S^1)^{n-3})=n-2$;  its genetic code is $\la\{n,n-3,n-4,\ldots,1\}\ra$.
Another case in which it is not true that $\TC(\Mbar(\lbar))\ge 2n-6$ is the case in which the genetic code is $\la\{ n\}\ra$. This space $\Mbar(\lbar)$
 is homeomorphic to $RP^{n-3}$, for which the TC is often much less than $2n-7$.(\cite[\S4.8]{F},\cite{D})

For $n=6$ and $n=7$, our analysis is essentially complete. Of the 20 equivalence classes of 6-gon spaces, the 18 for which we prove $\TC\ge 2n-6$ plus the torus and projective space is a complete list.
 Of the 134 equivalence classes of 7-gon spaces, we prove $\TC\ge 2n-6$ for 130 of them. Two others are the torus and projective space, and two are those listed in Theorem \ref{mainthm} for which our method does not imply $TC\ge 2n-6$.

 Of the 2469 equivalence classes of 8-gon spaces, our results only apply to 690. We feel that our work presents very strong evidence that almost all of them will satisfy $\TC\ge 2n-6$, provably by  cohomological methods. There will be a few exceptions, similarly to the case $n=7$. But more work needs to be done.

We will prove in Theorems \ref{first}, \ref{high}, \ref{cohthm}, \ref{norange}, and \ref{finale}, and in Sections \ref{33sec} and \ref{T1sec}, that, if $X=\Mbar(\lbar)$ with genetic codes of the types listed in Table \ref{gT}, then there are classes $v_1,\ldots,v_{2n-7}$ in $H^1(X)$ such that \begin{equation}\label{goal}\prod_{i=1}^{2n-7}(v_i\times1+1\times v_i)\ne0\in
H^*(X\times X).\end{equation} Theorem \ref{mainthm} then follows from \cite[Cor 4.40]{F}.
Here and throughout, all cohomology groups have coefficients in $\zt$.

\section{Proof of Theorem \ref{mainthm} for genetic codes with a  gene of size 2 or only one gene, which has size 3}\label{nkbsec}
 The first two and last cases of Table \ref{gT} are handled in this section. This also prepares us for more complicated analyses which appear later.

We will make much use of the following description of $H^*(\Mbar(\lbar))$, a slight reinterpretation of the result originally obtained in \cite[Cor 9.2]{HK}.
We introduce the term {\it subgee} for a subset $S$ which satisfies $S\le G$ for some gee $G$ of $\ell$. This is equivalent to saying that $S\cup\{ n\}$ is short, but is more directly related to the genetic code.
\begin{thm}  If $\lbar=(\ell_1,\ldots,\ell_n)$, then the ring $H^*(\Mbar(\lbar))$ is generated by classes $R,V_1,\ldots,V_{n-1}$ in $H^1(\Mbar(\lbar))$ subject to only the following relations:
\begin{enumerate}
\item All monomials of the same degree which are divisible by the same $V_i$'s are equal. Hence monomials $R^kV_S:=R^k\prod_{i\in S}V_i$ for $S\subset[n-1]$ span $H^*(\Mbar(\ell))$.
\item $V_S=0$ unless $S$ is a subgee of $\ell$.
\item If $S$ is a subgee with $|S|\ge n-2-d$, then,
\begin{equation}\label{(3)}\sum_{T\notint S}R^{d-|T|}V_T =0.\end{equation}
\end{enumerate}\label{relnsprop}\end{thm}
\ni Here the sum is taken over all $T$ for which $T$ does not intersect $S$; i.e., the two sets are disjoint.

From now on, let $m=n-3$ denote $\dim(\Mbar(\ell))$. Relations (\ref{(3)}) in $H^m(\Mbar(\ell))$ require $|S|\ge1$, while those in $H^{m-1}(\Mbar(\ell))$ require $|S|\ge2$. Let $\vbar=v\otimes1+1\otimes v$. We reinterpret (\ref{goal}) as wishing to find classes $v_i$ for which the component of\begin{equation}\label{goal2}\ol{v_1}\cdots\ol{v_{2m-1}}\ne0\in H^m(\Mbar(\ell))\otimes H^{m-1}(\Mbar(\ell)).\end{equation}
Note that the component in $H^{m-1}\ot H^m$ is symmetrical to this, and others are 0 for dimensional reasons.

The first case of (\ref{goal2}), and hence Theorem \ref{mainthm}, follows from the following result.
\begin{thm} If  $\lbar$ has genetic code $\la\{n,a\}\ra$ with $1\le a\le n-1$, and $X=\Mbar(\lbar)$, then
$$\ol{V_1}^{m}\ \ol{R}^{m-1}\ne0\in H^m(X)\otimes H^{m-1}(X).$$\label{first}
\end{thm}
\begin{proof} Since there are no gees of size $\ge2$, $H^{m-1}(X)$ has $\{R^{m-1},R^{m-2}V_1,\ldots,R^{m-2}V_a\}$ as basis, while in $H^m(X)$, the relations of type (\ref{(3)}) are
$$R^{m}+\sum^a_{\substack{j=1\\j\ne i}}R^{m-1}V_j=0$$
for $1\le i\le a$.
Subtracting pairs of relations reduces this set of relations to $R^{m-1}V_1=\cdots=R^{m-1}V_a$ and $R^{m}=(a-1)R^{m-1}V_1$.
 We obtain, in bidegree $(m,m-1)$,
\begin{eqnarray}&&(V_1\otimes 1+1\otimes V_1)^m(R\otimes 1+1\otimes R)^{m-1}\nonumber\\
&=&\sum\tbinom mi\tbinom{m-1}{m-i}V_1^iR^{m-i}\otimes V_1^{m-i}R^{i-1}\nonumber\\
&=&(\tbinom{2m-1}m+1)R^{m-1}V_1\otimes R^{m-2}V_1+R^{m-1}V_1\otimes R^{m-1}.\label{sum}\end{eqnarray}
Here we have used that $\sum\binom mi\binom{m-1}{m-i}=\binom{2m-1}m$ and noted that all nonzero terms $V_1^iR^{m-i}\otimes V_1^{m-i}R^{i-1}$ in the sum are equivalent to $R_1^{m-1}V_1\otimes R^{m-2}V_1$ except the one with $i=m$.
Since $\{R^{m-1},R^{m-2}V_1\}$ is linearly independent and $R_1^{m-1}V_1\ne0$, (\ref{sum}) is nonzero.
\end{proof}

The following result, with an amazingly simple proof, appears in the last line of Table \ref{gT}. It is our broadest result, but we have not been able to apply the method to other situations.
\begin{thm}\label{high} If the genetic code of  $\ell$ contains a gene $\{n,b\}$ and at least one other gene, and $X=\Mbar(\ell)$, then $\TC(X)\ge 2n-6$.\end{thm}
\begin{proof} If $G$ is another gee of $\ell$, then $G$ must contain at least two elements, all less than $b$, for otherwise one of $G$ and $\{b\}$ would be $\ge$ the other.
There is a homomorphism $\psi:H^{m-1}(X)\to\zt$ sending $R^{m-1}$ and $R^{m-2}V_b$ to 1, and all other monomials to 0. This is true because for every subgee $S$ containing at least two elements, these two monomials both appear in the sum in (\ref{(3)}).

For the Poincar\'e duality isomorphism $\phi:H^m(X)\to\zt$, there must be a nonempty subset $\{i_1,\ldots,i_t\}\subset[b-1]$ such that $\phi(R^{m-t}V_{i_1}\cdots V_{i_t})\ne0$. To see why, assume to the contrary. Then letting $S=\{b\}$ in (\ref{(3)}) implies that $\phi(R^m)=0$, and then letting $S$ in (\ref{(3)}) be one of the gees other than $\{b\}$ implies $\phi(R^{m-1}V_b)=0$, and hence $\phi$ would be identically zero, contradicting that it is an isomorphism.

Now let $R^{m-t}V_{i_1}\cdots V_{i_t}$ satisfy $\phi(R^{m-t}V_{i_1}\cdots V_{i_t})=1$ and let $\psi$ be as in the first paragraph of this proof. Then $$(\phi\otimes\psi)(\ol{V_{i_1}}^{m+1-t}\ \ol{V_{i_2}}\cdots\ol{V_{i_t}}\ \Rbar^{m-1})=\phi(V_{i_1}^{m+1-t}V_{i_2}\cdots V_{i_t})\psi(R^{m-1})=1,$$
because all other terms in the expansion have their second factor annihilated by $\psi$.\end{proof}

The case of (\ref{goal2}) and hence Theorem \ref{mainthm} corresponding to the second line of Table \ref{gT} follows immediately from the following result, the proof of which  will occupy the rest of this section.

\begin{thm}\label{cohthm} Let $X=\Mbar(\lbar)$ with genetic code $\la\{n,a+b,a\}\ra$ with $n>a+b>a>0$ and $n\ge6$. Let $m=n-3$ and
suppose $2^{e-1}<m\le 2^e$. Then in $H^*(X\times X)$,
\begin{enumerate}
\item[a.] if $a$ is even, then $\ol{V_1}^{2m-1-2^e}\ \ol{V_{a+b}}\ \Rbar^{2^e-1}\ne0$;
\item[b.] if $a$ is odd and $m\ne2^{e-1}+1$, then $\ol{V_1}^{m-1}\ \ol{V_{a+b}}^2\ \Rbar^{m-2}\ne0$;
\item[c.] if $a$ is odd and $m=2^{e-1}+1$, then $\ol{V_1}^m\ \ol{V_{a+b}}^{m-1}\ne0$.
\end{enumerate}
\end{thm}
As the powers of $R$ are essentially just place-keepers, we will usually omit writing them in the future. Also, we will often refer to $R$ itself as $V_0$.
By parts (1) and (2) of Theorem \ref{relnsprop}, both $H^m(X)$ and $H^{m-1}(X)$ are spanned by $V_i$, $0\le i\le a+b$, and $V_{i,j}$, $1\le i\le a$, $i<j\le a+b$.

All the classes $V_i$ with $i\le a$  play the same role in the relations (\ref{(3)}). The same is true of all the classes $V_i$ with $a<i\le a+b$.
The way that we will show a class $z$ in $H^{2m-1}(X\times X)$ is nonzero is by constructing a uniform homomorphism $\psi:H^{m-1}(X)\to\zt$ such that
$(\phi\ot\psi)(z)\ne0$, where $\phi:H^m(X)\to\zt$ is the Poincar\'e duality isomorphism. By {\it uniform homomorphism}, we mean one satisfying
\begin{itemize}
\item $\psi(V_i)=\psi(V_j)$ if $i,j\le a$ or if $a<i,j\le a+b$, and
\item $\psi(V_{i,j})=\psi(V_{i,k})$ if $j,k\le a$ or if $a<j,k\le a+b$.
\end{itemize}
(Here we are omitting writing the powers of $R$ accompanying the $V$'s.)

 We will let $Y_1$
refer to any $R^kV_i$ with $i\le a$, and $Y_2$ to any $R^kV_i$ with $a<i\le a+b$. Similarly, $Y_{1,1}$ denotes $R^kV_{i,j}$ with $i<j\le a$, while $Y_{1,2}$ is $R^kV_{i,j}$ with $i\le a<j\le a+b$.
If $\psi:H^{m-1}(X)\to \zt$  is a uniform homomorphism, then $\psi(Y_W)$ is a well-defined element of $\zt$ for each of the four possible subscripts  $W$.
We also let $w_1=V_i$ for $1\le i\le a$, and $w_2=V_i$ for $a<i\le a+b$. The difference between $w$ and $Y$ is that $w_i$ refers to the 1-dimensional class $V_i$, while $Y_S$ refers to $R^kV_S$ for an appropriate value of $k$.
Finally, an element of $[a+b]$ is of type 1 if it is $\le a$, and otherwise is of type 2.

The relations (\ref{(3)}) in $H^{m-1}(X)$ are of type $\br_{1,1}$ and $\br_{1,2}$, depending upon whether the subgee $S$ has both elements of type 1, or one element of each type. If $a=1$, there are no $\br_{1,1}$ relations since there are not two distinct positive integers $\le a$.
If $a>1$ and $\psi$ is a uniform homomorphism, then $\psi(\br_{1,1})$ is
\begin{equation}\label{r11}\psi(Y_0)+(a-2)\psi(Y_1)+b\psi(Y_2)+\tbinom{a-2}2\psi(Y_{1,1})+(a-2)b\psi(Y_{1,2})=0.\end{equation}
These coefficients count the number of relevant subsets $T$ in the sum. For example, in the last term, the number of subsets $T\subset[a+b]$ containing one type-1 element and one type-2 element which are disjoint from a given set $S$ which has two type-1 elements  is $(a-2)b$,
and $\psi$ sends each of them to the same element of $\zt$. Similarly, abbreviating $\psi(Y_W)$ as $\psi_W$,  $\psi(\br_{1,2})$ is the relation
\begin{equation}\label{r12}\psi_0+(a-1)\psi_1+(b-1)\psi_2+\tbinom{a-1}2\psi_{1,1}+(a-1)(b-1)\psi_{1,2}=0.\end{equation}

\begin{prop}\label{phiprop} Let $\phi:H^m(X)\to\zt$ be the Poincar\'e duality isomorphism, and let $\phi_W=\phi(Y_W)$. Then
\begin{eqnarray*}\phi_{1,1}=\phi_{1,2}&=&1\\
\phi_2&=&a-1\\
\phi_1&=&a+b\\
\phi_0&=&(a-1)b+\tbinom{a-1}2.\end{eqnarray*}\end{prop}
\begin{proof} By symmetry, $\phi$ is uniform. Using the notation introduced above, there are relations in $H^m(X)$ of the form $\br_1$ and $\br_2$ satisfying, respectively,
\begin{eqnarray*}\phi_0+(a-1)\phi_1+b\phi_2+\tbinom{a-1}2\phi_{1,1}+(a-1)b\phi_{1,2}&=&0\\
\phi_0+a\phi_1+(b-1)\phi_2+\tbinom a2\phi_{1,1}+a(b-1)\phi_{1,2}&=&0,\end{eqnarray*}
as well as relations $\br_{1,1}$ and $\br_{1,2}$ like (\ref{r11}) and (\ref{r12}), but with $\psi$ replaced by $\phi$.
As one can check by row-reduction or substitution, the nonzero solution of these four equations (mod 2) is the one stated in the proposition.\end{proof}

Next we expand the expressions in Theorem \ref{cohthm}. The parity of $a$ is not an issue in these expansions.
The expression in part (a) expands, in bidegree $(m,m-1)$, as (in our new notation)
\begin{eqnarray*}&&\sum_{i=1}^{2m-2-2^e}\tbinom{2m-1-2^e}i\tbinom{2^e-1}{m-i-1}w_1^iw_2R^{m-i-1}\otimes w_1^{2m-1-2^e-i}R^{2^e-m+i}\\
&&+w_2R^{m-1}\otimes w_1^{2m-1-2^e}R^{2^e-m}+w_1^{2m-1-2^e}w_2R^{2^e-m}\otimes R^{m-1}\\
&&+\sum_{i=1}^{2m-2-2^e}\tbinom{2m-1-2^e}i\tbinom{2^e-1}{m-i}w_1^iR^{m-i}\otimes w_1^{2m-1-2^e-i}w_2R^{2^e-m+i-1}\\
&&+\tbinom{2^e-1}mR^m\otimes w_1^{2m-1-2^e}w_2R^{2^e-m-1}+\tbinom{2^e-1}{2^e+1-m}w_1^{2m-1-2^e}R^{2^e+1-m}\otimes w_2R^{m-2}.\end{eqnarray*}
The first line is $\dsum_{i=1}^{2m-2-2^e}\tbinom{2m-1-2^e}i\tbinom{2^e-1}{m-i-1}$ times $Y_{1,2}\otimes Y_1$. This sum is easily seen to be 0 mod 2.
Similarly the sum on the third line is $\equiv0$. We obtain that the expansion in part (a) is, in bidegree $(m,m-1)$,
\begin{equation}\label{expa} Y_2\otimes Y_1+Y_{1,2}\otimes R^{m-1}+(1+\delta_{m,2^e})R^m\otimes Y_{1,2}+Y_1\otimes Y_2.\end{equation}
We frequently use Lucas's Theorem for evaluation of mod 2 binomial coefficients. For example, here we use that $\binom{2^e-1}i\equiv1$ for all nonnegative $i\le 2^e-1$.

Part (b) expands, in bidegree $(m,m-1)$, as
\begin{eqnarray*}&&\sum_{i=1}^{m-2}\tbinom{m-1}i\tbinom{m-2}{m-i-2}w_1^iw_2^2R^{m-i-2}\otimes w_1^{m-1-i}R^i\\
&&+\sum_{i=2}^{m-2}\tbinom{m-1}i\tbinom{m-2}{m-i}w_1^iR^{m-i}\otimes w_1^{m-1-i}w_2^2R^{i-2}\\
&&+w_2^2R^{m-2}\otimes w_1^{m-1}+mw_1^{m-1}R\times w_2^2R^{m-3}.\end{eqnarray*}
For $m\ne 2^{e-1}+1$, $$\sum_{i=1}^{m-2}\tbinom{m-1}i\tbinom{m-2}{m-i-2}\equiv\tbinom{2m-3}{m-2}+1\equiv1$$ and $$\sum_{i=2}^{m-2}\tbinom{m-1}i\tbinom{m-2}{m-i}\equiv\tbinom{2m-3}m+m\equiv\delta_{m,2^e}+m,$$ and so, similarly to (\ref{sum}), the expansion equals
\begin{equation}\label{expb}Y_{1,2}\otimes Y_1+(\delta_{m,2^e}+m)Y_1\otimes Y_{1,2}+Y_2\otimes Y_1+mY_{1}\otimes Y_2.\end{equation}

The expansion of part (c) of Theorem \ref{cohthm} is easier. It equals, in bidegree $(m,m-1)$,
\begin{equation}\label{expc} w_1w_2^{2^{e-1}}\ot w_1^{2^{e-1}}+w_1^{2^{e-1}+1}\ot w_2^{2^{e-1}}=Y_{1,2}\otimes Y_1+Y_1\otimes Y_2.\end{equation}

Now we show, one-at-a-time, that there are uniform homomorphisms $\psi$ such that $\phi\ot\psi$ sends (\ref{expa}), (\ref{expb}), and (\ref{expc}) to 1.

If $\psi:H^{m-1}(X)\to\zt$ is a uniform homomorphism,
applying $\phi\ot\psi$ to (\ref{expa}) with $a$ even yields, using Proposition \ref{phiprop},
\begin{equation}\label{expr}\psi(Y_1)+\psi(Y_0)+\eps_1\psi(Y_{1,2})+b\psi(Y_2).\end{equation}
Here and in the following, $\eps_t$ denotes an element of $\zt$ whose value turns out to be irrelevant. To have (\ref{expr}) be nonzero, we need a uniform homomorphism $\psi$ satisfying the system with the following augmented matrix. The columns represent
$\psi(Y_0)$, $\psi(Y_1)$, $\psi(Y_2)$, $\psi(Y_{1,1})$, and $\psi(Y_{1,2})$, respectively, and the second and third rows are (\ref{r11}) and (\ref{r12}).
$$\left[\begin{array}{ccccc|c}1&1&b&0&\eps_1&1\\
1&0&b&\eps_2&0&0\\
1&1&b-1&\eps_2&b-1&0\end{array}\right]$$
Here we have noted that since $a$ is even, $\binom{a-2}2\equiv\binom{a-1}2$ mod 2.
This system is easily seen to have a solution, proving part (a) of Theorem \ref{cohthm}.

Applying $\phi\ot\psi$ to (\ref{expb}) with $a$ odd and $a>1$ similarly yields
$$\psi(Y_1)+\eps_3\psi(Y_{1,2})+m(b+1)\psi(Y_2).$$
Now $\psi$ must satisfy the following system, which is also easily seen to have a solution.
$$\left[\begin{array}{ccccc|c}0&1&m(b+1)&0&\eps_3&1\\
1&1&b&\eps_4&b&0\\
1&0&b-1&\eps_4+1&0&0\end{array}\right]$$
Here we have used that if $a$ is odd, then $\binom{a-1}2\equiv\binom{a-2}2+1$.
If $a=1$, then the fourth column and second row are removed, and again there is a solution.

Finally, applying $\phi\ot\psi$ to (\ref{expc}) with $a$ odd leads to the following system for $\psi$, which again has a solution.
$$\left[\begin{array}{ccccc|c}0&1&b+1&0&0&1\\
1&1&b&\eps_4&b&0\\
1&0&b-1&\eps_4+1&0&0\end{array}\right]$$
Again, the fourth column and second row are omitted if $a=1$, but there is still a solution.
This completes the proof of Theorem \ref{cohthm}.

\section{Proof of \ref{mainthm} for Type 2 cases}
In this section we prove Theorem \ref{mainthm} in the 27 cases which are marked $T_2$ in Table \ref{gT}. This is easier than most of what we are doing because in each case we are dealing with a single set of gees, rather than one with parameters, such as the $a$ and $b$ in Theorem \ref{cohthm}.
So subscripts refer to actual $V$, rather than to a range of $V$'s.

\begin{thm}\label{norange}If $X=\Mbar(\ell)$, where $\ell=(\ell_1,\ldots,\ell_n)$, has gees as listed in the first half of any of the rows of Table \ref{Tt}, and $m=n-3\ge4$, then the component of
\begin{equation}\label{N1}\ol{V_1}^{m-1}\ \ol{V_2}^2\ \ol{V_3}\ \Rbar^{m-3}\ne0\in H^m(X)\otimes H^{m-1}(X),\end{equation}
if $m-1$ is not a 2-power or it is  the final case in the table. If $m-1$ is a 2-power and it is not the last case of Table \ref{Tt}, then \begin{equation}\label{N2}\ol{V_1}^{m}\ \ol{V_2}^2\ \ol{V_3}\ \Rbar^{m-4}\ne0.\end{equation}
\end{thm}

In the following table, we list the cases in which $\psi_S=\psi(Y_S)=1$ for the homomorphism $\psi:H^{m-1}(X)\to\zt$ which we use.

\begin{tab}\label{Tt}

\begin{center}
{\scalefont{1}{
$$\renewcommand\arraystretch{1.0}\begin{array}{l|l}
\text{gees}&\psi_S\text{ which equal 1}\\
\hline
321,\ 42&\psi_1,\psi_{1,3},\psi_{1,4}\\
321,\ 42,\ 51&\psi_1,\psi_{1,3},\psi_{1,4}\\
321,\ 42,\  61&\psi_1,\psi_{1,3},\psi_{1,4}\\
321,\ 43&\psi_1,\psi_{1,2},\psi_{1,3},\psi_{1,4}\\
321,\ 43,\ 51&\psi_1,\psi_{1,2},\psi_{1,3},\psi_{1,4}\\
321,\ 43,\ 52&\psi_0,\psi_1,\psi_2,\psi_5,\psi_{1,2},\psi_{1,3},\psi_{1,4},\psi_{2,5}\\
321,\ 43,\ 52,\ 61&\psi_0,\psi_1,\psi_2,\psi_5,\psi_{1,2},\psi_{1,3},\psi_{1,4},\psi_{2,5}\\
321,\ 43,\ 61&\psi_1,\psi_{1,2},\psi_{1,3},\psi_{1,4}\\
321,\ 43,\ 62&\psi_1,\psi_{1,2},\psi_{1,3},\psi_{1,4},\psi_{1,5},\psi_{1,6}\\
321,\ 52&\psi_0,\psi_1,\psi_2,\psi_5,\psi_{1,3},\psi_{1,4},\psi_{2,5}\\
321,\ 52,\ 61&\psi_0,\psi_1,\psi_2,\psi_5,\psi_{1,3},\psi_{1,4},\psi_{2,5}\\
321,\ 53&\psi_1,\psi_2,\psi_3,\psi_{1,4},\psi_{2,4},\psi_{3,4}\\
321,\ 53,\ 61&\psi_1,\psi_2,\psi_3,\psi_{1,4},\psi_{2,4},\psi_{3,4}\\
321,\ 53,\ 62&\psi_1,\psi_2,\psi_3,\psi_{1,4},\psi_{2,4},\psi_{3,4}\\
321,\ 62&\psi_0,\psi_1,\psi_2,\psi_4,\psi_{1,4},\psi_{2,4},\psi_{3,4}\\
321,\ 63&\psi_1,\psi_2,\psi_3,\psi_{1,4},\psi_{2,4},\psi_{3,4}\\
421,\ 43&\psi_1,\psi_{1,2},\psi_{1,3},\psi_{1,4}\\
421,\ 43,\ 51&\psi_1,\psi_{1,2},\psi_{1,3},\psi_{1,4}\\
421,\ 43,\ 52&\psi_0,\psi_1,\psi_2,\psi_5,\psi_{1,2},\psi_{1,3},\psi_{1,4},\psi_{2,5}\\
421,\ 43,\ 52,\ 61&\psi_0,\psi_1,\psi_2,\psi_5,\psi_{1,2},\psi_{1,3},\psi_{1,4},\psi_{2,5}\\
421,\ 43,\ 61&\psi_1,\psi_{1,2},\psi_{1,3},\psi_{1,4}\\
421,\ 43,\ 62&\psi_1,\psi_{1,2},\psi_{1,3},\psi_{1,4},\psi_{1,5},\psi_{1,6}\\
421,\ 52&\psi_0,\psi_1,\psi_2,\psi_5,\psi_{1,2},\psi_{1,3},\psi_{1,4},\psi_{2,5}\\
421,\ 52,\ 61&\psi_0,\psi_1,\psi_2,\psi_5,\psi_{1,3},\psi_{1,4},\psi_{2,5}\\
421,\ 62&\psi_0,\psi_1,\psi_2,\psi_5,\psi_{1,5},\psi_{2,5}\\
521,\ 62&\psi_1,\psi_{1,3},\psi_{1,4},\psi_{1,5},\psi_{1,6}\\
43,\ 52,\ 61&\psi_1,\psi_{1,6}
\end{array}$$}}
\end{center}
\end{tab}

Using Proposition \ref{21prop}, $\phi\ot\psi$ applied to the product in (\ref{N1}) equals
\begin{eqnarray}\label{ee1}&&(\phi_{2,3}+\phi_{1,2,3})\psi_1+\phi_{1,3}(\psi_2+\psi_{1,2})\\
\nonumber &+&(m-1)\phi_1(\psi_{2,3}+\psi_{1,2,3})+\begin{cases}\phi_1\psi_{1,2,3}&\text{if $m$ is a 2-power}\\
\phi_{1,2,3}\psi_1+\phi_{1,3}\psi_{1,2}&\text{if $m-1$ is a 2-power}\\
0&\text{otherwise.}\end{cases}\end{eqnarray}
 When $m-1$ is a 2-power, $\phi\ot\psi$ applied to (\ref{N2}) equals
\begin{equation}\label{ee2}\phi_1(\psi_{2,3}+\psi_{1,2,3})+\phi_{1,2,3}\psi_1+\phi_{1,3}\psi_{1,2}.\end{equation}

The homomorphism $\phi:H^m(X)\to\zt$ in each of the 27 cases is easily determined by {\tt Maple}. We define a matrix  whose columns (representing monomials) correspond to all the subgees, including 0 (the empty set), and whose rows (representing relations (\ref{(3)})) correspond to all subgees except 0. An entry  in the matrix is 1 iff the  subgees for its row and column labels are disjoint. We row reduce and read off the unique nonzero solution. There is a nice pattern for all $\phi_S$ when $|S|\ge2$, but we have not discerned a pattern for $\phi_i$. What is relevant for the first 26 cases is that we always have $\phi_{1,2,3}=1$, while $\phi_{1,3}=\phi_{2,3}=0$.

Each case admits several homomorphisms $\psi:H^{m-1}\to\zt$. They must satisfy a system of homogeneous equations whose matrix is like the one for $\phi$ except that the rows correspond only to subgees with more than one element. In Table \ref{Tt}, we have listed, for each case, one such $\psi$ that works. They all satisfy that $\psi_1=1$ and $\psi_{2,3}=\psi_{1,2,3}=0$. We see that  (\ref{ee1})  equals 1 for the first 26 cases if $m-1$ is not a 2-power, as does (\ref{ee2}) if $m-1$ is a 2-power.

For the final (27th) case in the table, there is a homomorphism $\psi$ for which the only nonzero values are $\psi_1$ and $\psi_{1,6}$. Also $\phi_{i,j}=1$ for all $\{i,j\}$. Of course, $\phi_{1,2,3}=0$, since $V_{1,2,3}=0$ in this case. Thus we
obtain that (\ref{ee1}) is nonzero, due to $\phi_{2,3}\psi_1$.

The {\tt Maple} program that computed $\phi$ and verified $\psi$ for these 27 cases can be viewed at \cite{27}.

\section{Proof of \ref{mainthm} for genetic codes with a single gene of size 4}\label{4sec}
In this section,  $X$ denotes $\Mbar(\lbar)$ where the genetic code of $\lbar$ is $\la\{n,a+b+c,a+b,a\}\ra$ with $n>a+b+c>a+b>a>0$.  We prove (\ref{goal2}) for $X$  by an analysis  similar to that for $\la\{n,a+b,a\}\ra$ in Theorem \ref{cohthm}, except that there are many more cases to consider.

We adopt notation similar to that of the proof of Theorem \ref{cohthm}, using $w_1$, $w_2$, and $w_3$ for $V_i$ with $i$ in $[1,a]$, $(a,a+b]$, and $(a+b,a+b+c]$, respectively.
Similarly, $Y_S$ for a multiset $S$ containing elements of $\{1,2,3\}$ refers to monomials whose $V_i$'s have the types of the elements of $S$.

We begin by proving, similarly to Proposition \ref{phiprop},
\begin{thm}\label{fourthm}  Let $\phi:H^m(X)\to\zt$ be the Poincar\'e duality isomorphism, and let $\phi_W=\phi(Y_W)$. Then
\begin{eqnarray*}\phi_{1,1,1}=\phi_{1,1,2}=\phi_{1,1,3}&=&1\\
\phi_{1,2,2}=\phi_{1,2,3}&=&1\\
\phi_{2,2}=\phi_{2,3}&=&a-1\\
\phi_{1,3}&=&a+b\\
\phi_{1,1}=\phi_{1,2}&=&a+b+c-1\\
\phi_3&=&(a-1)(b-1)+\tbinom a2\\
\phi_2&=&(a-1)(b+c)+\tbinom a2\\
\phi_1&=&(a-1)(a+b+c-1)+\tbinom{a-1}2+\tbinom b2+(b-1)(c-1)\\
\phi_0&=&\tbinom a2(a+b+c-1)+(a-1)(\tbinom b2+(b-1)(c-1)).
\end{eqnarray*}
\end{thm}
\begin{proof} Let $\cS=$
$$\{\emptyset,(1),(2),(3),(1,1),(1,2),(1,3),(2,2),(2,3),(1,1,1),(1,1,2),(1,1,3),(1,2,2),(1,2,3)\}$$ denote the set of types of elements that can be subgees. For example, $(1,2)$ refers to a subgee $\{i,j\}$ with $i\le a$ and $a<j\le a+b$.
For $U\in \cS$ and $1\le i\le 3$, let $u_i$ denote the number of $i$'s in $U$. For example, if $U=(1,1,2)$, then $u_1=2$, $u_2=1$, and $u_3=0$. For each $U\in\cS-\emptyset$, there is a relation $\br_U$ of type (\ref{(3)}), and $\phi(\br_U)$ is
\begin{equation}\label{Ueq}\sum_{U'\in\cS}\tbinom{a-u_1}{u_1'}\tbinom{b-u_2}{u_2'}\tbinom{c-u_3}{u_3'}\phi_{U'}=0.\end{equation}
For example, (\ref{r12}) is of this form, including only $a$ and $b$ (not $c$) and corresponding to $U=(1,2)$, and with $\psi$ instead of $\phi$.
The set of all equations (\ref{Ueq}) is a system of 13 homogeneous equations over $\zt$ in 14 unknowns $\phi_{U'}$, and its nonzero solution is the one stated in the theorem.

This solution was found manually by row reduction, and then checked by {\tt Maple}, noting that the system only depends on $a$ mod 4, $b$ mod 4, and $c$ mod 2. The program verified that the solution worked in all 32 cases.

There are special considerations when $a=1$ or 2, or $b=1$. For example, if $b=1$, then $\tbinom{b-2}1$ and $\tbinom{b-2}2$ should be 0 for us, but are 1 in most binomial coefficient formulas. But the relations $\br_U$ in which such coefficients appear are not present, and so the equations (\ref{Ueq}) for these $U$ need not be considered.
\end{proof}

Let \begin{equation}\label{cs'}\cS'=\{(1,1),(1,2),(1,3),(2,2),(2,3),(1,1,1),(1,1,2),(1,1,3),(1,2,2),(1,2,3)\}\end{equation}
correspond to the types of gees of size $\ge2$.
For $\psi:H^{m-1}(X)\to\zt$ to be a uniform homomorphism with $\psi_W=\psi(Y_W)$, the conditions that must be satisfied are, for $U\in\cS'$,
\begin{equation}\label{Upeq}\sum_{U'\in\cS}\tbinom{a-u_1}{u_1'}\tbinom{b-u_2}{u_2'}\tbinom{c-u_3}{u_3'}\psi_{U'}=0.\end{equation}

 To prove (\ref{goal2}), we seek $\psi:H^{m-1}(X)\to\zt$ satisfying (\ref{Upeq}) for all $U\in\cS'$, and $\{v_1,\ldots,v_{2m-1}\}$ such that \begin{equation}\label{Dwi}(\phi\ot\psi)(\ol{v_1}\cdots\ol{v_{2m-1}})=1.\end{equation}
  The classes $v_i\in H^1(X)$ that we will use depend on the mod 4 values of $a$ and $b$, and $c$ mod 2. First we deal with two cases that turn out to be exceptional.
The following easily-verified lemma will be useful.
\begin{lem} $\binom A2+AB+AC+BC+\binom B2\equiv0$ mod 2 iff $A+B\equiv0\mod4$ or $A+B+2C\equiv1\mod4$.\label{techlem}\end{lem}

The following result is our first verification of (\ref{Dwi}). The last two exceptions in Theorem \ref{mainthm} are due to the requirement here that $m>4$ (to make $m-6+\eps\ge0$).
\begin{prop}\label{prop111} Suppose $m>4$, $a\equiv b\equiv1\mod 4$, and $c\equiv1\mod 2$. The isomorphism $\phi:H^m(X)\to\zt$ sends each $Y_{i,j,k}$ to $1$ and other monomials to 0.
There exists a uniform homomorphism $\psi:H^{m-1}(X)\to\zt$ sending each $Y_{i,j}$ to 1 and other monomials to 0. Let $\eps=1$ if $m-2$ is a 2-power, and $\eps=2$ otherwise. Then
$$(\phi\otimes\psi)(\ol{w_1}^{m-\eps}\ \ol{w_2}^2\ \ol{w_3}^3\ \Rbar^{m-6+\eps})=1.$$
\end{prop}
\begin{proof} The $\phi$-part is easily checked using Theorem \ref{fourthm}, and the $\psi$-part follows from Lemma \ref{techlem}. Indeed, the LHS of (\ref{Upeq}) becomes
$$\tbinom{a-u_1}2+(a-u_1)(b-u_2)+(a-u_1)(c-u_3)+\tbinom{b-u_2}2+(b-u_2)(c-u_3),$$
which, by \ref{techlem}, is 0 for the prescribed $a$, $b$, $c$ if $u_1+u_2\equiv2\mod4$ or $u_1+u_2+2u_3\equiv3\mod4$, and this is easily verified to be true for the ten elements of $\cS'$.

In the expansion of $\ol{w_1}^{m-\eps}\ \ol{w_2}^2\ \ol{w_3}^3\ \Rbar^{m-6+\eps}$, there are no terms with repeated subscripts in either $Y$ factor since it only involves one element of each type. Also, there are no $Y_{1,2,3}\otimes Y_{1,2}$ or $Y_{1,2,3}\ot Y_{2,3}$ terms, since $\ol{w_2}^2=w_2^2\ot1+1\ot w_2^2$. When $\eps=2$, the $Y_{1,2,3}\ot Y_{1,3}$ terms come from
\begin{eqnarray*}&&\sum_{i=1}^{m-3}\tbinom{m-2}i\tbinom{m-4}{m-i-4}w_1^iw_2^2w_3^2R^{m-i-4}\ot w_1^{m-2-i}w_3R^i\\
&+&\sum_{i=1}^{m-3}\tbinom{m-2}i\tbinom{m-4}{m-i-3}w_1^iw_2^2w_3R^{m-i-3}\ot w_1^{m-2-i}w_3^2R^{i-1}\\
&=&(\tbinom{2m-6}{m-4}+1)Y_{1,2,3}\ot Y_{1,3}=Y_{1,2,3}\ot Y_{1,3}\end{eqnarray*}
since $m-2$ is not a 2-power. If $m-2$ is a 2-power, the similar calculation, involving $\sum\binom{m-1}i\binom{m-5}{m-i-4}$ and $\sum\binom{m-1}i\binom{m-5}{m-i-3}$, gives just $\binom{2m-6}{m-4}=1$.
\end{proof}

The other exceptional case verifying (\ref{Dwi}) is similar.
\begin{prop} Suppose $a\equiv2\mod4$, $b\equiv4\mod4$, and $c\equiv1\mod2$. The isomorphism $\phi:H^m(X)\to\zt$ sends $Y_{i,j,k}$, $Y_{2,2}$, and $Y_{2,3}$ to $1$,  and other monomials to 0.
There exists a uniform homomorphism $\psi:H^{m-1}(X)\to\zt$ sending each $Y_{i,j}$ to 1 and other monomials to 0. Let $\eps=1$ if $m-2$ is a 2-power, and $\eps=2$ otherwise. Then
$$(\phi\otimes\psi)(\ol{w_1}^2\ \ol{w_2}^2\ \ol{w_3}^{m-\eps}\ \Rbar^{m-5+\eps})=1.$$
\label{prop241}
\end{prop}
\begin{proof} The only term in the expansion which is mapped nontrivially is $Y_{2,3}\otimes Y_{1,3}$. It appears as
$$\sum_{i=1}^{m-\eps-1}\tbinom{m-\eps}i\tbinom{m-5+\eps}{m-2-i}w_2^2w_3^iR^{m-2-i}\ot w_1^2w_3^{m-\eps-i}R^{i+\eps-3}$$
with coefficient $\binom{2m-5}{m-2}+\tbinom{m-5+\eps}{m-2}+\tbinom{m-5+\eps}{\eps-2}=1$.\end{proof}

Let $\abar$ (resp.~$\bbar$) denote the mod 4 value of $a$ (resp.~$b$), and $\cbar$ the mod 2 value of $c$.
For the other 30 cases of $\abar$, $\bbar$, and $\cbar$  (or 90 if you consider deviations regarding whether $m$ or $m-1$ is a 2-power),
we use {\tt Maple} to tell that an appropriate $\psi$ can be found. To accomplish this in all cases, several choices for the exponents of $\ol{w_1}$, $\ol{w_2}$, and $\ol{w_3}$ are required. Possibly some choice of exponents might work in all cases, but we did not find one.

Most of our results will be obtained using the following result.
\begin{prop} If neither $m$ nor $m-1$ is a 2-power, then the component of $\ol{w_1}^\a\ \ol{w_2}^2\ \ol{w_3}\ \Rbar^{2m-4-\a}$ in bidegree $(m,m-1)$ equals
\begin{eqnarray*}&&\tbinom{2m-4-\a}m(Y_0\otimes Y_{1,2,3}+Y_1\otimes Y_{1,2,3})\\
&+&\tbinom{2m-4-\a}{m-1}(Y_{1,2,3}\otimes Y_0+Y_{1,2,3}\otimes Y_1+Y_3\otimes Y_{1,2}+Y_{1,3}\otimes Y_{1,2})\\
&+&\tbinom{2m-4-\a}{m-2}(Y_2\otimes Y_{1,3}+Y_{1,2}\otimes Y_3)\\
&+&\tbinom{2m-4-\a}{m-3}(Y_{2,3}\otimes Y_1+Y_{1,2,3}\otimes Y_1+Y_{1,3}\otimes Y_2+Y_{1,3}\otimes Y_{1,2})\\
&+&\tbinom{2m-4-\a}{m-4}(Y_1\otimes Y_{2,3}+Y_1\otimes Y_{1,2,3}).\end{eqnarray*}\label{21prop}
If $m$ is a 2-power, there is an additional term $Y_1\ot Y_{1,2,3}$. If $m-1$ is a 2-power, $Y_{1,2,3}\ot Y_1+Y_{1,3}\ot Y_{1,2}$ must be added to the above expansion.
\end{prop}
\begin{proof} $\ol{w_1}^\a\ \ol{w_2}^2\ \ol{w_3}\ \Rbar^{2m-4-\a}$ in bidegree $(m,m-1)$ expands as
\begin{eqnarray*}&&\sum_{i=0}^\a\tbinom{\a}i\tbinom{2m-4-\a}{m-i-3}w_1^iw_2^2w_3R^{m-i-3}\otimes w_1^{\a-i}R^{m-1-\a+i}\\
&+&\sum_{i=0}^\a\tbinom{\a}i\tbinom{2m-4-\a}{m-i}w_1^iR^{m-i}\otimes w_1^{\a-i}w_2^2w_3R^{m-4-\a+i}\\
&+&\sum_{i=0}^\a\tbinom{\a}i\tbinom{2m-4-\a}{m-i-2}w_1^iw_2^2R^{m-i-2}\otimes w_1^{\a-i}w_3R^{m-2-\a+i}\\
&+&\sum_{i=0}^\a\tbinom{\a}i\tbinom{2m-4-\a}{m-i-1}w_1^iw_3R^{m-i-1}\otimes w_1^{\a-i}w_2^2R^{m-3-\a+i}.\end{eqnarray*}
 If neither $m$ nor $m-1$ is a 2-power, the coefficients $\tbinom{2m-4}{m-t}$ for $t=3$, 0, 2, 1, which occur as the sum of all coefficients on a line, are 0. Thus the four lines equal, respectively,
\begin{eqnarray*}&&\tbinom{2m-4-\a}{m-3}(Y_{2,3}\otimes Y_1+Y_{1,2,3}\otimes Y_1)+\tbinom{2m-4-\a}{m-\a-3}(Y_{1,2,3}\otimes Y_0+Y_{1,2,3}\otimes Y_1),\\
&&\tbinom{2m-4-\a}m(Y_0\otimes Y_{1,2,3}+Y_1\otimes Y_{1,2,3})+\tbinom{2m-4-\a}{m-\a}(Y_1\otimes Y_{2,3}+Y_1\otimes Y_{1,2,3}),\\
&&\tbinom{2m-4-\a}{m-2}(Y_2\otimes Y_{1,3}+Y_{1,2}\otimes Y_{1,3})+\tbinom{2m-4-\a}{m-\a-2}(Y_{1,2}\otimes Y_3+Y_{1,2}\otimes Y_{1,3}),\\
&&\tbinom{2m-4-\a}{m-1}(Y_3\otimes Y_{1,2}+Y_{1,3}\otimes Y_{1,2})+\tbinom{2m-4-\a}{m-\a-1}(Y_{1,3}\otimes Y_2+Y_{1,3}\otimes Y_{1,2}).
\end{eqnarray*}
The sum of these is easily manipulated into the claimed form. If $m$ is a 2-power, then $\binom{2m-4}m$ is odd, while if $m-1$ is a 2-power, $\binom{2m-4}{m-1}$ and $\binom{2m-4}{m-3}$ are odd, yielding the additional terms in the claim.
\end{proof}

The following result follows immediately from Proposition \ref{21prop} and Theorem \ref{fourthm}.
\begin{cor}\label{cor1}
Let $q_t=\binom{2m-4-\a}{m-t}$ for $0\le t\le 4$, and $\psi_W=\psi(Y_W)$, where $\psi:H^{m-1}(X)\to\zt$ is a uniform homomorphism. Then, if neither $m$ nor $m-1$ is a 2-power,
\begin{eqnarray}&&(\phi\otimes\psi)(\ol{w_1}^\a\ \ol{w_2}^2\ \ol{w_3}\ \Rbar^{2m-4-\a})\label{bigpsi}\\
&=&q_1\psi_0+(q_1+q_3a)\psi_1+q_3(a+b)\psi_2+q_2(a+b+c-1)\psi_3\nonumber\\
&&+(q_1(ab+1+\tbinom a2)+q_3(a+b))\psi_{1,2}+q_2((a-1)(b+c)+\tbinom a2)\psi_{1,3}\nonumber\\
&&+q_4((a-1)(a+b+c-1)+\tbinom{a-1}2+\tbinom b2+(b-1)(c-1))(\psi_{2,3}+\psi_{1,2,3})\nonumber\\
&&+q_0((a+\tbinom a2-1)(a+b+c-1)+\tbinom{a-1}2+a\tbinom b2+a(b-1)(c-1))\psi_{1,2,3}\nonumber.
\end{eqnarray}
\end{cor}

\begin{lem} If $m=2^e+m'$ with $2\le m'\le 2^e-1$ and $\a=2m'-3$ and $0\le t\le4$, then $\binom{2m-4-\a}{m-t}\equiv1$ mod 2.\label{lem1}\end{lem}
\begin{proof} The top of the binomial coefficient is $2^{e+1}-1$, while the bottom is $\le 2^{e+1}-1$.\end{proof}

The first of several verifications of (\ref{goal2}) for multiple values of $(\abar,\bbar,\cbar)$ appears in the following result.
\begin{thm}\label{thm1} If $m=2^e+m'$ with $2\le m'\le 2^e-1$, then
$$\ol{w_1}^{2m'-3}\ \ol{w_2}^2\ \ol{w_3}\ \Rbar^{2^{e+1}-1}\ne0\in H^{2m-1}(X\times X)$$
for the values of $\abar$, $\bbar$, and $\cbar$ which have an $\times$ in the \ref{thm1} column of Table \ref{bigtable}.
\end{thm}
\begin{proof} This is the case described in Lemma \ref{lem1}, so that $q_0=\cdots=q_4=1$ in (\ref{bigpsi}). We need values of $\psi_W$ such that the RHS of (\ref{bigpsi}) equals 1, and (\ref{Upeq}) is satisfied for all $U\in\cS'$. (Recall that the relationship of $U$ to (\ref{Upeq}) is that $u_i$ is the number of occurrences of $i$ in $U$.) Altogether this is 11 equations over $\zt$ in 14 unknowns. The coefficients of the equations depend only on $\abar$, $\bbar$, and $\cbar$. It is a simple matter to run {\tt Maple} on these 32 cases, and it tells us that there is a solution in exactly the claimed cases. The {\tt Maple} program, input and output, for this and the remaining cases of this section can be seen at \cite{maple}. The two cases, $(\abar,\bbar,\cbar)=(1,1,1)$ and $(2,4,1)$, considered in Propositions \ref{prop111} and \ref{prop241} are not included in Table \ref{bigtable} because they did not yield a solution in any of the situations whose results appear as a column of that table. The special situation when $a=1$ or 2 or $b=1$ is not a problem, exactly as in the proof of Theorem \ref{fourthm}.
\end{proof}

The next result is very similar. The only difference is a small change in the exponent of $\ol{w_1}$ (and hence also of $\Rbar$).
This changes the values of $q_t$.
\begin{thm}\label{thm2} If $m=2^e+m'$ with $2\le m'\le 2^e-1$, then
$$\ol{w_1}^{2m'-2}\ \ol{w_2}^2\ \ol{w_3}\ \Rbar^{2^{e+1}-2}\ne0\in H^{2m-1}(X\times X)$$
for the values of $\abar$, $\bbar$, and $\cbar$ which have an $\times$ in the \ref{thm2} column of Table \ref{bigtable}.
\end{thm}
\begin{proof} In this case, $q_t=m-t-1$. That is the only change from the proof of Theorem \ref{thm1}. Here we require that a solution must exist both when $\qbar=(1,0,1,0,1)$ and $(0,1,0,1,0)$, covering both parities of $m$. Here and later $\qbar=(q_0,q_1,q_2,q_3,q_4)$.\end{proof}

The third result also just involves a change in the exponent of $\ol{w_1}$. This time $q_t=\binom{m-t+2}2$, so we require a solution for all four of the vectors $\qbar$, corresponding to mod 4 values of $m$.
\begin{thm}\label{thm3} If $m=2^e+m'$ with $2\le m'\le 2^e-1$, then
$$\ol{w_1}^{2m'-1}\ \ol{w_2}^2\ \ol{w_3}\ \Rbar^{2^{e+1}-3}\ne0\in H^{2m-1}(X\times X)$$
for the values of $\abar$, $\bbar$, and $\cbar$ which have an $\times$ in the \ref{thm3} column of Table \ref{bigtable}.
\end{thm}

We can prove (\ref{goal2}) in the remaining case $(\abar,\bbar,\cbar)=(1,2,1)$  by changing the exponent of $\ol{w_2}$. We begin with the following analogue of Proposition \ref{21prop}, whose proof is totally analogous.
\begin{prop} If neither $m$ nor $m-1$ is a 2-power, then the component of $\ol{w_1}^\a\ \ol{w_2}\ \ol{w_3}\ \Rbar^{2m-3-\a}$ in bidegree $(m,m-1)$ equals
\begin{eqnarray*}&&\tbinom{2m-3-\a}m(Y_0\otimes Y_{1,2,3}+Y_1\otimes Y_{1,2,3})\\
&+&\tbinom{2m-3-\a}{m-1}(Y_{1,2,3}\otimes Y_0+Y_{1,2,3}\otimes Y_1+Y_3\otimes Y_{1,2}+Y_{1,3}\otimes Y_{1,2}+Y_{2}\otimes Y_{1,3}+Y_{1,2}\otimes Y_{1,3})\\
&+&\tbinom{2m-3-\a}{m-2}(Y_{2,3}\otimes Y_1+Y_{1,2,3}\otimes Y_1+Y_{1,3}\otimes Y_2+Y_{1,3}\otimes Y_{1,2}+Y_{1,2}\otimes Y_3+Y_{1,2}\otimes Y_{1,3})\\
&+&\tbinom{2m-3-\a}{m-3}(Y_1\otimes Y_{2,3}+Y_1\otimes Y_{1,2,3}).\end{eqnarray*}\label{11prop}
If $m$ is a 2-power, there is an additional $Y_1\ot Y_{1,2,3}$.
\end{prop}
\ni We do not need to use this proposition when $m-1$ is a 2-power.

\begin{thm}\label{thm4} If $m=2^e+m'$ with $2\le m'\le 2^e-1$, then
$$\ol{w_1}^{2m'-1}\ \ol{w_2}\ \ol{w_3}\ \Rbar^{2^{e+1}-2}\ne0\in H^{2m-1}(X\times X)$$
for the values of $\abar$, $\bbar$, and $\cbar$ which have an $\times$ in the \ref{thm4} column of Table \ref{bigtable}.
\end{thm}
\begin{proof}
Let $q'_t=\binom{2m-3-\a}{m-t}$ for $0\le t\le 3$. Using Proposition \ref{11prop} and Theorem \ref{fourthm}, we find that
\begin{eqnarray}&&(\phi\otimes\psi)(\ol{w_1}^\a\ \ol{w_2}\ \ol{w_3}\ \Rbar^{2m-3-\a})\label{111line}\\
&=&q'_1\psi_0+(q'_1+q'_2a)\psi_1+q'_2(a+b)\psi_2+q'_2(a+b+c-1)\psi_3\nonumber\\
&&+(q'_1(ab+1+\tbinom a2)+q'_2(a+b))\psi_{1,2}\nonumber\\
&&+(q'_1(a(b+c)+a-1+\tbinom a2)+q'_2(a+b+c-1))\psi_{1,3}\nonumber\\
&&+q'_3((a-1)(a+b+c-1)+\tbinom{a-1}2+\tbinom b2+(b-1)(c-1))(\psi_{2,3}+\psi_{1,2,3})\nonumber\\
&&+q'_0((a+\tbinom a2-1)(a+b+c-1)+\tbinom{a-1}2+a\tbinom b2+a(b-1)(c-1))\psi_{1,2,3}\nonumber.
\end{eqnarray}

Similarly to Lemma \ref{lem1}, with $\a=2m'-1$, we have
 $\binom{2m-3-\a}{m-t}\equiv m-t-1\mod 2$, and the result follows similarly to the three previous ones, having {\tt Maple} check whether there is a solution to the system of 11 equations in 14 unknowns, whose first equation is that the RHS of (\ref{111line}) equals 1 and others are, as before, (\ref{Upeq})  for each $U\in\cS'$. This time a solution is required for both $\qbar'=(0,1,0,1)$ and $(1,0,1,0)$.
 \end{proof}

 Referring to Table \ref{bigtable} and Theorems \ref{thm1}, \ref{thm2}, \ref{thm3}, and \ref{thm4}, accompanied by Propositions \ref{prop111} and \ref{prop241}, we find that, when neither $m$ nor $m-1$ is a 2-power, (\ref{goal2}) is satisfied for all $(\abar,\bbar,\cbar)$, establishing Theorem
 \ref{mainthm} when neither $m$ nor $m-1$ is a 2-power. Next we handle the case when $m=2^e$.

 \begin{thm}\label{thm5} If $m=2^e$, then, for $1\le\eps\le3$,
$$\ol{w_1}^{2^e-\eps}\ \ol{w_2}^2\ \ol{w_3}\ \Rbar^{2^e+\eps-4}\ne0\in H^{2m-1}(X\times X)$$
for the values of $\abar$, $\bbar$, and $\cbar$ which have an $\times$ in the \ref{thm5}$(\eps)$ column of Table \ref{bigtable}.
\end{thm}
\begin{proof}  Because of the change due to $m=2^e$ noted in Proposition \ref{21prop}, the expression in Corollary \ref{cor1} has an extra $((a-1)(a+b+c-1)+\binom{a-1}2+\binom b2+(b-1)(c-1))\psi_{1,2,3}$ added. The vectors $\qbar$ are $(0,1,1,1,1)$, $(0,0,1,0,1)$, and $(0,0,0,1,1)$ for
$\eps=3$, $2$, $1$, respectively. The other 10 equations for the $\psi_W$'s are as before. {\tt Maple} tells us when the system has a solution.
\end{proof}

The next result is the $2^e$-analogue of Theorem \ref{thm4}. As shown in Table \ref{bigtable}, this, Theorem \ref{thm5}, and Propositions \ref{prop111} and \ref{prop241} imply (\ref{goal2}) when $m=2^e$.
 \begin{thm}\label{thm6} If $m=2^e$, then
$$\ol{w_1}^{2^e-1}\ \ol{w_2}\ \ol{w_3}\ \Rbar^{2^e-2}\ne0\in H^{2m-1}(X\times X)$$
for the values of $\abar$, $\bbar$, and $\cbar$ which have an $\times$ in the \ref{thm6} column of Table \ref{bigtable}.
\end{thm}
\begin{proof}  Because of the change due to $m=2^e$ noted in Proposition \ref{11prop}, the expression in (\ref{111line}) has an extra $((a-1)(a+b+c-1)+\binom{a-1}2+\binom b2+(b-1)(c-1))\psi_{1,2,3}$ added. The vector $\qbar'$ is  $(0,0,1,0)$,
and the other 10 equations for the $\psi_W$'s are as before. {\tt Maple} tells us when the system has a solution.
\end{proof}

Finally, we handle the case $m=2^e+1$.
 \begin{thm}\label{thm7} If $m=2^e+1$, then, for $\eps=\pm1$,
$$\ol{w_1}^{2^e+\eps}\ \ol{w_2}^2\ \ol{w_3}\ \Rbar^{2^e-\eps-2}\ne0\in H^{2m-1}(X\times X)$$
for the values of $\abar$, $\bbar$, and $\cbar$ which have an $\times$ in the \ref{thm7}$(\eps)$ column of Table \ref{bigtable}.
\end{thm}
\begin{proof}  Because of the change due to $m=2^e+1$ noted in Proposition \ref{21prop}, the expression in Corollary \ref{cor1} has an extra  $\psi_1+(a+b)\psi_{1,2}$  added. The vectors $\qbar$ are $(0,0,0,0,1)$ and $(0,0,1,1,1)$ for
$\eps=1$ and $-1$, respectively. The other 10 equations for the $\psi_W$'s are as before. {\tt Maple} tells us when the system has a solution.
\end{proof}

\bigskip
\begin{tab}\label{bigtable}

\begin{center}
{\scalefont{.9}{
$$\renewcommand\arraystretch{1.3}\begin{array}{ccc|cccc|cccc|cc}
&&&\multicolumn{4}{c}{m\ne2^e,2^e+1}&\multicolumn{4}{c}{m=2^e}&\multicolumn{2}{c}{m=2^e+1}\\
\abar&\bbar&\cbar&\ref{thm1}&\ref{thm2}&\ref{thm3}&\ref{thm4}&\ref{thm5}(3)&\ref{thm5}(2)&\ref{thm5}(1)&\ref{thm6}&\ref{thm7}(1)&\ref{thm7}(-1)\\
\hline
1&1&2&&&\xt&&&&\xt&&\xt&\\
1&2&1&&&&\xt&&&&\xt&\xt&\xt\\
1&2&2&&&\xt&\xt&&&\xt&\xt&\xt&\\
1&3&1&&\xt&\xt&\xt&&\xt&\xt&\xt&\xt&\xt\\
1&3&2&&\xt&\xt&\xt&&\xt&\xt&\xt&\xt&\xt\\
1&4&1&&\xt&\xt&\xt&&\xt&\xt&\xt&\xt&\\
1&4&2&&\xt&\xt&\xt&&\xt&\xt&\xt&\xt&\xt\\
\hline
2&1&1&\xt&&&&\xt&\xt&\xt&\xt&\xt&\xt\\
2&1&2&\xt&&&&\xt&&\xt&\xt&\xt&\\
2&2&1&&\xt&&&\xt&\xt&\xt&\xt&\xt&\xt\\
2&2&2&&&\xt&\xt&\xt&\xt&\xt&\xt&\xt&\xt\\
2&3&1&\xt&&&&\xt&&&&\xt&\xt\\
2&3&2&\xt&&&&\xt&\xt&\xt&\xt&\xt&\xt\\
2&4&2&&\xt&&\xt&&\xt&&\xt&\xt&\\
\hline
3&1&1&&\xt&\xt&\xt&&\xt&\xt&\xt&\xt&\xt\\
3&1&2&&&\xt&\xt&\xt&\xt&\xt&\xt&\xt&\\
3&2&1&\xt&\xt&\xt&\xt&\xt&&\xt&\xt&\xt&\xt\\
3&2&2&\xt&\xt&&\xt&\xt&\xt&&\xt&\xt&\xt\\
3&3&1&&\xt&&&&\xt&&&&\xt\\
3&3&2&&\xt&\xt&\xt&\xt&\xt&\xt&\xt&\xt&\xt\\
3&4&1&\xt&\xt&&\xt&\xt&\xt&\xt&\xt&\xt&\xt\\
3&4&2&\xt&\xt&&\xt&\xt&\xt&\xt&\xt&\xt&\xt\\
\hline
4&1&1&\xt&\xt&\xt&\xt&\xt&\xt&\xt&\xt&\xt&\xt\\
4&1&2&\xt&&\xt&\xt&\xt&&\xt&\xt&\xt&\xt\\
4&2&1&\xt&\xt&\xt&\xt&\xt&\xt&&&&\xt\\
4&2&2&\xt&\xt&&\xt&\xt&\xt&&\xt&\xt&\xt\\
4&3&1&\xt&\xt&\xt&\xt&\xt&\xt&\xt&\xt&\xt&\xt\\
4&3&2&\xt&&\xt&\xt&\xt&\xt&\xt&\xt&\xt&\xt\\
4&4&1&\xt&\xt&\xt&\xt&\xt&\xt&\xt&\xt&\xt&\xt\\
4&4&2&\xt&\xt&\xt&\xt&\xt&\xt&\xt&\xt&\xt&\xt
\end{array}$$}}
\end{center}
\end{tab}

The {\tt Maple} program that performed all these verifications can be viewed at \cite{maple}.
We conclude that
\begin{thm}\label{finale} If $X=\Mbar(\lbar)$ with genetic code $\la\{n,a+b+c,a+b,a\}\ra$, then (\ref{Dwi}) holds for some uniform homomorphism $\psi$ and some set of $2m-1$ classes $v_i$.\end{thm}
\begin{proof} Table \ref{bigtable} shows that for all $(\abar,\bbar,\cbar)$ except $(1,1,1)$ and $(2,4,1)$ one of the tabulated theorems applies, while the two exceptional cases are covered in
Propositions \ref{prop111} and \ref{prop241}.\end{proof}

\section{Proof when there are two genes, each of size 3}\label{33sec}
In this section, we prove the case of Theorem \ref{mainthm} corresponding to the fourth line of Table \ref{gT}.
In this case, $\ell$ has gees $\{a+b+c,a+b\}$ and $\{a+b+c+d,a\}$, with $a,b,c,d\ge1$. Any other way of having two gees of size 2 would have one $\ge$ the other, which is not allowed in a genetic code.

Let $X=\Mbar(\ell)$. Notation is similar to that of previous proofs. For $i=1$, 2, 3, and 4, $w_i$ refers to $V_t$ with $t$ in the intervals $[1,a]$, $(a,a+b]$, $(a+b,a+b+c]$, and $(a+b+c,a+b+c+d]$, respectively. The $Y_S$ is a monomial $R^kV_S$, where the subscripts $S$ are of the indicated types.
The multiset $S$ can have at most two elements. For example, $Y_{2,2}=R^kV_iV_j$ with $a<i<j\le a+b$.

\begin{prop} The Poincar\'e duality isomorphism $\phi:H^m(X)\to\zt$ satisfies
\begin{eqnarray*}\phi_{1,1}=\phi_{1,2}=\phi_{1,3}=\phi_{1,4}=\phi_{2,2}=\phi_{2,3}&=&1\\
\phi_4&=&a+1\\
\phi_3&=&a+b+1\\
\phi_2&=&a+b+c\\
\phi_1&=&a+b+c+d\\
\phi_0&=&\tbinom{a-1}2+\tbinom b2+bc+(a+1)(b+c+d).
\end{eqnarray*}
There is a uniform homomorphism $\psi:H^{m-1}(X)\to \zt$ satisfying $\psi_{1,1}=\psi_{1,2}=\psi_{1,3}=\psi_{1,4}=\psi_{2,2}=\psi_{2,3}=0$,
$\psi_1=\psi_2=\psi_3=\psi_4=1$, and $\psi_0=a+b+c+d$.
\end{prop}
Here, as before, $\phi_S=\phi(Y_S)$ and $\psi_S=\psi(Y_S)$.
\begin{proof}
The isomorphism $\phi$ is the nonzero homomorphism whose images $\phi_S$ are the nonzero solution of the homogeneous system with matrix given in Table \ref{matrix}. Uniform homomorphisms $\psi$ must satisfy the homogeneous system with matrix given by the last six rows of Table \ref{matrix}. These are the analogues of (\ref{Ueq}) and (\ref{Upeq}).
Both solutions are easily verified.

\bigskip
\begin{tab}\label{matrix}
\def\tb{\tbinom}
\begin{center}
{\scalefont{.7}{
$$\renewcommand\arraystretch{1.2}\begin{array}{ccccccccccc}
0&1&2&3&4&1,1&1,2&1,3&1,4&2,2&2,3\\
\hline
1&a-1&b&c&d&\tb{a-1}2&(a-1)b&(a-1)c&(a-1)d&\tb b2&bc\\
1&a&b-1&c&d&\tb a2&a(b-1)&ac&ad&\tb{b-1}2&(b-1)c\\
1&a&b&c-1&d&\tb a2&ab&a(c-1)&ad&\tb b2&b(c-1)\\
1&a&b&c&d-1&\tb a2&ab&ac&a(d-1)&\tb b2&bc\\
1&a&b&c&d&\tb{a-2}2&ab&ac&ad&\tb b2&bc\\
1&a-1&b-1&c&d&\tb{a-1}2&(a-1)(b-1)&(a-1)c&(a-1)d&\tb{b-1}2&(b-1)c\\
1&a-1&b&c-1&d&\tb{a-1}2&(a-1)b&(a-1)(c-1)&(a-1)d&\tb b2&b(c-1)\\
1&a-1&b&c&d-1&\tb{a-1}2&(a-1)b&(a-1)c&(a-1)(d-1)&\tb b2&bc\\
1&a&b&c&d&\tb a2&ab&ac&ad&\tb{b-2}2&bc\\
1&a&b-1&c-1&d&\tb a2&a(b-1)&a(c-1)&ad&\tb{b-1}2&(b-1)(c-1)
\end{array}$$}}
\end{center}
\end{tab}
Since these matrices are mod 2, we have often written $a-2$ as $a$.
\end{proof}

The following analogue of Lemma \ref{lem1} is easily checked.
\begin{lem} If $m=2^e+m'$ with $2\le m'\le 2^e+1$ and $\a=2m'-3$ and $2\le t\le4$, then $\binom{2m-4-\a}{m-t}\equiv1$ mod 2.\label{lem1p}\end{lem}
Now Proposition \ref{21prop} yields that for $m$ as in Lemma \ref{lem1p}
$$(\phi\ot\psi)(\ol{w_1}^{2m'-3}\ \ol{w_2}^2\ \ol{w_3}\ \Rbar^{2^{e+1}-1})=\phi_{1,2}\psi_3+\phi_{2,3}\psi_1+\phi_{1,3}\psi_2=1,$$
establishing the result.

\section{Proof in Type 1 situations}\label{T1sec}
In this section, we prove the case of Theorem \ref{mainthm} corresponding to the fifth line of Table \ref{gT}, what we call the ``Type 1'' cases. These are the cases where there are gees
$\{1+b+c,1+b,1\}$ and $\{1+b+c+d,1\}$ with $b,c,d\ge1$. The proof here is similar to that in Section \ref{4sec}, requiring several different products, depending on the nature of the parameters, and also requiring {\tt Maple} to just check for us that there exists a suitable uniform homomorphism $\psi:H^{m-1}(X)\to \zt$, rather than giving a nice explicit formula for $\psi$. Here, as usual, $X=\Mbar(\ell)$ and $m=n-3$, where $\ell=(\ell_1,\ldots,\ell_n)$ has the prescribed gees.
The variables $w_i$ for $1\le i\le4$ correspond to $V_t$ with, respectively, $t=1$ and $t$ in intervals $(1,1+b]$, $(1+b,1+b+c]$, and $(1+b+c,1+b+c+d]$, and $Y_S$ is defined according to these subscripts in the usual way. Note that there are no $Y_{1,1}$ variables as there usually are, since there are not two distinct variables with subscript $\le1$.

Our {\tt Maple} program uses a matrix similar to that of Table \ref{matrix} with $a=1$. The column labels are exactly those that appear in the following proposition.
\begin{prop} For $X$ as just described, the Poincar\'e duality isomorphism $\phi:H^m(X)\to\zt$ satisfies, with $\phi_S=(\phi(Y_S)$
\begin{eqnarray*}\phi_{1,4}=\phi_{1,2,2}=\phi_{1,2,3}&=&1\\
\phi_{1,3}&=&b+1\\
\phi_{1,2}&=&b+c\\
\phi_1&=&\tbinom b2+(b+1)(c+1)+d\\
\phi_0=\phi_2=\phi_3=\phi_4=\phi_{2,2}=\phi_{2,3}&=&0.\end{eqnarray*}\label{23prop}
\end{prop}

\begin{prop} For the $X$ just described, there is a uniform homomorphism $\psi:H^{m-1}(X)\to\zt$ satisfying
\begin{equation}\label{pop}(\phi\otimes\psi)(\ol{w_1}^{m-1}\ \ol{w_2}^2\ \ol{w_3}\ \Rbar^{m-3})=1\end{equation}
except in the following situations:

\begin{itemize}
\item $b\equiv1$ mod 4, $c$ odd, and $d$ even,
\item $m-1$ is a 2-power, or
\item $m$ is a 2-power, $b\equiv2$ mod 4, $c$ odd, and $d$ even.
\end{itemize}

\end{prop}
\begin{proof} By Propositions \ref{21prop} and \ref{23prop}, if neither $m$ nor $m-1$ is a 2-power, then (\ref{pop}) equals
\begin{equation}\label{ppp}\psi_1+(b+1)(\psi_2+\psi_{1,2})+(m-1)(\tbinom b2+(b+1)(c+1)+d)(\psi_{2,3}+\psi_{1,2,3}).\end{equation}
The relations that $\psi$ must satisfy, in addition to making (\ref{ppp}) equal to 1, are that for each of the subscripts $S$ of length 2 or 3 in Proposition \ref{23prop}
the analogue of (\ref{Upeq}) (or the analogue of the last six rows of Table \ref{matrix}, with $a=1$) must equal 0. This row reduction can be done by hand, finding that a solution
exists unless $b\equiv1$ mod 4, $c$ odd, and $d$ even. If $m$ is a 2-power, then (\ref{ppp}) becomes
$$\psi_1+(b+1)(\psi_2+\psi_{1,2})+(\tbinom b2+(b+1)(c+1)+d)\psi_{2,3},$$
and we find that a solution making this equal to 1 exists unless $b\equiv1$ or 2 mod 4, $c$ is odd, and $d$ even.

If $m-1$ is a 2-power, we use $(\phi\ot\psi)(\ol{w_1}^m\,\ol{w_2}^2\,\ol{w_3}\ \Rbar^{m-4})$. Using Proposition \ref{21prop} again, the expression which must equal 1 is
$$\psi_1+(b+1)\psi_{1,2}+(\tbinom b2+(b+1)(c+1)+d)(\psi_{2,3}+\psi_{1,2,3}),$$
and this can be achieved, consistent with the relations, unless $b\equiv1$ mod 4, $c$ odd, and $d$ even.

When $m>4$, $b\equiv1$ mod 4, $c$ odd, and $d$ even, we use the product $\ol{w_1}^{m-\eps}\ \ol{w_2}^2\ \ol{w_3}^3\ \Rbar^{m-6+\eps}$ used in Proposition \ref{prop111}.
The only nonzero value $\phi_S$ relevant to this product is $\phi_{1,2,3}$. One easily checks that there is a uniform homomorphism $\psi$ whose only nonzero values are
$\psi_{1,2}$ and $\psi_{1,3}$. Because of the $\ol{w_2}^2$, there can be no $Y_{1,2,3}\ot Y_{1,2}$ term in the expansion, and, as in the proof of \ref{prop111}, the coefficient of $Y_{1,2,3}\ot Y_{1,3}$ is nonzero.

When $m=4$, $b\equiv1$ mod 4, $c$ odd, and $d$ even, the nonzero $\phi$-values are $\phi_{1,4}$, $\phi_{1,2,2}$, and $\phi_{1,2,3}$, and there is a uniform homomorphism $\psi$ on $H^{m-1}(X)$ whose nonzero values are $\psi_0$, $\psi_4$, $\psi_{2,2}$, and $\psi_{2,3}$. Then $(\phi\ot\psi)(\ol{w_1}^2\ \ol{w_4}^3\, \Rbar^2)=\phi_{1,4}\psi_4=1 $.

Finally, if $m=2^e$, $b\equiv2$ mod 4, $c$ odd, and $d$ even, then the nonzero $\phi$-values are $\phi_1$, $\phi_{1,2}$, $\phi_{1,3}$, $\phi_{1,4}$, $\phi_{1,2,2}$, and $\phi_{1,2,3}$, and there is a uniform homomorphism $\psi$ with nonzero values $\psi_2=\psi_3=\psi_4=\psi_{2,2}=\psi_{2,3}=1$. Then $(\phi\ot\psi)(\ol{w_1}^{2^e}\ \ol{w_4}\ \Rbar^{2^e-2})=\phi_1\psi_4=1$.
\end{proof}

 \def\line{\rule{.6in}{.6pt}}

\end{document}